\documentclass[10pt, a4paper]{amsart}

\usepackage[english]{babel}
\usepackage{amssymb}
\usepackage{amsmath}
\usepackage{enumerate}
\usepackage{epsfig}
\usepackage[all]{xy}
\usepackage{longtable}
\usepackage{textcomp}
\usepackage{float}
\usepackage{hyperref}

\newtheorem{Thm}{Theorem}[section]
\newtheorem{Lemma}[Thm]{Lemma}
\newtheorem{Prop}[Thm]{Proposition}
\newtheorem{Cor}[Thm]{Corollary}
\newtheorem*{RecallThm}{Theorem}

\theoremstyle{definition}
\newtheorem{Def}[Thm]{Definition}
\newtheorem{Ex}[Thm]{Example}
\newtheorem*{Notation}{Notations \& Conventions}

\theoremstyle{remark}
\newtheorem{Rmk}[Thm]{Remark}

\numberwithin{equation}{section}

\newcommand{\PP}{\mathbb{P}}
\newcommand{\A}{\mathbb{A}}
\newcommand{\C}{\mathbb{C}}
\newcommand{\R}{\mathbb{R}}

\newcommand{\cL}{\mathcal{L}}
\newcommand{\cO}{\mathcal{O}}
\newcommand{\cF}{\mathcal{F}}
\newcommand{\cI}{\mathcal{I}}
\newcommand{\cV}{\mathcal{V}}

\newcommand{\mult}{\operatorname{mult}}
\newcommand{\Sing}{\operatorname{Sing}}
\newcommand{\Cone}{\operatorname{Cone}}

\newcommand{\Exc}{\operatorname{Exc}}
\newcommand{\id}{\operatorname{id}}
\newcommand{\red}{\text{red}}
\newcommand{\st}{\middle|}

\begin{document} 

\title[Crepant resolutions and non-Kodaira fibres]{Crepant resolutions of Weierstrass threefolds and non-Kodaira fibres}

\author{Andrea Cattaneo}
\address{Andrea Cattaneo, Dipartimento di Matematica ed Informatica, Universit\`a degli Stu\-di di Parma, Parco Area delle Scienze 53/A, 43124 Parma, Italy}
\email{andrea.cattaneo@unipr.it}

\thanks{The author is granted with a research fellowship by Istituto Nazionale di Alta Matematica INdAM, and is supported by the Project PRIN ``Variet\`a reali e complesse: geometria, topologia e analisi armonica'', by the Project FIRB ``Geometria Differenziale e Teoria Geometrica delle Funzioni'', and by GNSAGA of INdAM. He wants to gratefully acknowledge all the people who helped him during the preparation of this paper, in particular M.\ Esole, M.\ Mella and B.\ van Geemen for interesting discussions and comments.}


\begin{abstract}
In this paper we want to study the non-Kodaira fibres in a smooth equidimensional elliptic threefold. If the morphism to the Weierstrass model of the fibration is crepant, then we can locate the non-Kodaira fibres and give a description of their structure. In particular, they lie over the singular points of the (reduced) discriminant of the fibration and are contraction of a Kodaira fibre, whose type can be predicted using Tate's algorithm.
\end{abstract}



\maketitle

\section*{Introduction}

The study of elliptic fibrations, i.e.\ morphisms $\pi: X \longrightarrow B$ whose generic fibre is an elliptic curve, dates back to Kodaira's paper \cite{Kodaira}, where one can find a detailed description of elliptic surfaces. Since then, the case of elliptic surfaces has been a benchmark for generalizations to higher dimension. Much of what is true for surfaces generalizes to the higher dimensional case, but many of these results are weaker: for example, what is true \emph{punctually} on the base curve in the case of surfaces becomes true only \emph{for the generic point of each irreducible component of a suitable codimension $1$ locus} for higher dimensional varieties.\\
An important result of Kodaira is the classification of the singular fibres which can occur in a minimal elliptic surface (cf.\ \cite[Thm. 6.2]{Kodaira}), which we call \emph{Kodaira fibres}. In this case there is a generalization to the higher dimensional case, and Kodaira's classification works over the generic point of each irreducible component of the discriminant locus of the fibration, but there appear also new non-Kodaira fibres: see e.g.\ \cite{Miranda}, \cite{GrassiMorrison-nonKodaira} or the more recent \cite{CCvG}, \cite{EsoleYau} and \cite{EsoleYauD5}.\\
Despite the great abundance of examples of non-Kodaira fibres in the previous papers, less is known on their structure and a complete classification is still far away. In this paper we will show that for the elliptic fibrations which are crepant resolutions of their Weierstrass model there is still a link with the Kodaira fibres: in this case the non-Kodaira fibres are contraction of Kodaira fibres, a fact which was pointed out by Miranda in some particular case in \cite[$\S$14]{SME3-f}.

The structure of the paper is as follows: in Section~\ref{sect: gor and can sings} and Section~\ref{section: elliptic fibration} we will recall the basic definitions and facts concerning the theory of singularities and elliptic fibrations respectively which will be used in the subsequent Sections. These two Sections are mainly expository, and we use them for references to keep this paper as self-contained as possible. Our main references are \cite{C3-f}, \cite{Pagoda}, \cite{YoungPerson}, \cite{Kollar-Mori} for Section~\ref{sect: gor and can sings} and \cite{Nakayama}, \cite{Miranda}, \cite{SME3-f} for Section~\ref{section: elliptic fibration}.

In Section~\ref{sect: crepant resolutions of weierstrass threefolds} we will address to the main topic of this paper: the study of the structure of non-Kodaira fibres in a smooth equidimensional elliptic threefold which is a crepant resolution of its Weierstrass model. As it will be explained in the introduction of Section~\ref{sect: crepant resolutions of weierstrass threefolds}, this is a class of threefolds it is quite natural to deal with, especially in some application. Our main Theorem (Theorem~\ref{thm: main theorem}) gives a partial classification of non-Kodaira fibres of a smooth equidimensional elliptic threefold in the context of biregular geometry:

\begin{RecallThm}
Let $\pi: X \longrightarrow B$ be a smooth equidimensional elliptic threefold, with morphism to the Weierstrass model $f: X \longrightarrow W$. If $f$ is a crepant resolution of $W$, then
\begin{enumerate}
\item $W$ has only $cDV$ singularities;
\item $f$ can not be factored as
\[\xymatrix{X \ar[dr]_\varphi \ar[rr]^f & & W\\
 & X' \ar[ur]_{f'} & }\]
where $\varphi$ is the contraction of some divisor $D \subseteq X$ and $X'$ is smooth;
\item if $X_b$ is a non-Kodaira fibre, then $b \in \Sing (\Delta(\pi)_\red)$;
\item if $X_b$ is a non-Kodaira fibre, then $X_b$ is a contraction of a Kodaira fibre, whose type can be determined by Tate's algorithm.
\end{enumerate}
\end{RecallThm}

In \cite{SME3-f}, Miranda gave a classification of all the possible non-Kodaira fibres of a smooth equidimensional elliptic threefold up to birational isomorphisms. Our main Theorem is compatible with Miranda's results (cf.\ \cite[$\S$14]{SME3-f}). The birational geometry of equidimensional elliptic threefolds was studied also in \cite{Grassi-Equidimensional}.\\
Finally, in Section~\ref{sect: examples} we will give some examples, some of which were the original motivation for this paper.

\begin{Notation}
Through all the paper, we will use the following notations.
\begin{enumerate}
\item All the varieties are defined over $\C$.
\item We will use a solid arrow $Y \longrightarrow X$ to denote \emph{morphisms} of varieties, and a dashed arrow $Y \dashrightarrow X$ to denote \emph{rational maps}. Hence, a \emph{birational morphism} is a morphism $Y \longrightarrow X$ which need not to be an isomorphism, but which has a birational inverse $X \dashrightarrow Y$.
\item For an elliptic fibration (see Definition~\ref{def: elliptic fibration}) $\pi: X \longrightarrow B$, we will always assume that $B$ is smooth, and we will mainly be interested in the case where $X$ is smooth as well. We will slightly abuse the notation saying that an elliptic fibration $\pi: X \longrightarrow B$ is smooth or singular if its total space $X$ is.
\item Given a vector bundle $\cV$ over $B$, we will denote by $\PP(\cV)$ the projective bundle of lines in $\cV$.
\end{enumerate}
\end{Notation}

\tableofcontents

\section{Gorenstein and canonical singularities}\label{sect: gor and can sings}

In this Section we will recall some definitions and results from singularity theory, in particular the definitions and properties of rational Gorenstein singularities (Section~\ref{sect: rational gorenstein sings}) and of canonical singularities (Section~\ref{sect: canonical singularities}). Finally, we will recall some useful facts on threefolds with $cDV$ singularities and their resolutions (Section~\ref{sect: resolving the du val locus}) which will be used in Section~\ref{sect: crepant resolutions of weierstrass threefolds}.

We start recalling the definition of Gorenstein varieties, which is the class of varieties we will deal with.

\begin{Def}
Let $X$ be a projective variety. We say that $X$ is \emph{Gorenstein} if it is Cohen--Macaulay, and its dualizing sheaf $\omega_X$ is locally free, i.e.\ it is a line bundle.
\end{Def}

In terms of a canonical Weil divisor $K_X$ for $X$, the condition that $\omega_X$ is locally free is equivalent to the condition that $K_X$ is a Cartier divisor.

\begin{Ex}\label{ex: gorenstein varieties}
Let $X$ be a smooth variety, and $Y \subseteq X$ a complete intersection. Then $Y$ is Gorenstein, since it is Cohen--Macaulay by \cite[Cor.\ III.4.5]{Altman-Kleiman} and $\omega_Y$ is a line bundle by \cite[Thm.\ III.7.11]{HAG}.
\end{Ex}

We will make a frequent use of the following definitions from singularity theory.

\begin{Def}\label{def: sing theory}
Let $X$ and $Y$ be varieties with $X$ normal, and $f: Y \longrightarrow X$ a morphism.
\begin{enumerate}
\item We say that $f$ is a \emph{partial resolution} if it is a proper birational morphism, with $Y$ normal.
\item We say that $f$ is a \emph{resolution} of $X$ if $f$ is a partial resolution with $Y$ smooth. If $X$ is singular, a resolution of $X$ is called also a \emph{resolution of the singularities}.
\item If $f$ is a partial resolution, we say that it is \emph{crepant} if $f^* \omega_X = \omega_Y$.
\item The \emph{exceptional locus} $\Exc f$ of a birational morphism $f$ is the (closed) subset of points $y \in Y$ where $f$ is not an isomorphism at $f(y)$. If $f$ is a partial resolution, we say that it is \emph{small} if $\Exc f$ has codimension at least $2$ in $Y$.
\item If $f$ is a resolution, we say that it is \emph{minimal} if any other resolution $f': Y' \longrightarrow X$ of $X$ factors through $f$.
\end{enumerate}
\end{Def}

\subsection{Rational Gorenstein singularities}\label{sect: rational gorenstein sings}

Now we want to address our attention to the singularities of Gorenstein varieties: the idea is that since they are close to smooth varieties, their singularities should be mild.

\begin{Def}[{cf.\ \cite[Def.\ 2.4]{C3-f}}]
Let $X$ be a Gorenstein variety. We say that $X$ has \emph{rational Gorenstein singularities} if there exists a resolution $f: Y \longrightarrow X$ of $X$ such that $f_* \omega_Y = \omega_X$.\\
We say that $P \in X$ is an \emph{elliptic Gorenstein singularity} if there exists a resolution $f: Y \longrightarrow X$ such that $f_* \omega_Y = \cI_P \cdot \omega_X$, where $\cI_P$ is the ideal sheaf defining $P$.
\end{Def}

Such singularities were studied in \cite[$\S$2]{C3-f}, where the following Proposition is proved.

\begin{Prop}[{cf.\ \cite[Thm.\ 2.6]{C3-f}}]\label{prop: properties of rational gorenstein singularities}
Let $X$ be an $n$-dimensional Gorenstein variety, with $n \geq 2$, and $P \in X$.
\begin{enumerate}
\item If $P$ is a rational Gorenstein point, then for a general hyperplane section $H$ through $P$, $P \in H$ is elliptic or rational Gorenstein.
\item If there exists a hyperplane section $H$ through $P$ such that $P \in H$ is a rational Gorenstein singularity, then $P \in X$ is a rational Gorenstein singularity.
\end{enumerate}
\end{Prop}

In the rest of this Section we want to give some examples of such singularities, in particular when $X$ is a surface or a threefold.

\begin{Ex}[Rational Gorenstein surface singularities]\label{ex: surface rational gorenstein singularities}
Let $X$ be a projective surface and $P \in X$ be a point. Then the following are equivalent (cf.\ \cite{Dufree}):
\begin{enumerate}
\item $P$ is a rational Gorenstein singularity.
\item There exists a resolution of singularities $f: Y \longrightarrow X$ which is crepant in a neighbourhood of $P$.
\item In a neighbourhood of $P$ we have that $X$ is analytically isomorphic to a \emph{Du Val singularity}, i.e.\ one of the hypersurface singularities of $\A^3$ defined by $f(x, y, z) = 0$, where $f$ is an equation from Table~\ref{table: du val singularities}.
\end{enumerate}

\begin{center}
\begin{longtable}{|c|c|}
\caption{List of Du Val singularities}
\label{table: du val singularities}\\
\hline
Name & Equation\\
\hline
\hline
$A_n$ & $x^2 + y^2 + z^{n + 1} \quad (n \geq 1)$\\
\hline
$D_n$ & $x^2 + y^2 z + z^{n - 1} \quad (n \geq 4)$\\
\hline
$E_6$ & $x^2 + y^3 + z^4$\\
\hline
$E_7$ & $x^2 + y^3 + y z^3$\\
\hline
$E_8$ & $x^2 + y^3 + z^5$\\
\hline
\end{longtable}
\end{center}

The names of the Du Val singularities refer to the Dynkin diagrams. In fact in the minimal resolution of a Du Val singularity, the exceptional divisors have the corresponding Dynkin diagram as incidence graph. The minimal resolution of a Du Val singularity can be obtained by a sequence of blow ups in the singular points, which ends as soon as the blown up surface becomes smooth. The exceptional curves introduced are all $(-2)$-curves.
\end{Ex}

\begin{Ex}[Elliptic Gorenstein surface singularities]\label{ex: surface elliptic gorenstein singularities}
Let $P$ be an elliptic Gorenstein singularity of the surface $X$, and denote by $f: Y \longrightarrow X$ its minimal resolution. Then $f^{-1}(P) = \cup_{i = 1}^m A_i$, and we can associate to $f$ a (unique) cycle $Z$ in $Y$ which is effective, satisfies $Z \cdot A_i \leq 0$ for all $i = 1, \ldots, m$ and which is minimal with respect to these two properties. Such a cycle is called the \emph{fundamental cycle} of $f$, and can be computed as follows (compare \cite[p.\ 1259]{Laufer}): start with $Z_1 = A_{i_1}$ arbitrary, and then define inductively $Z_j = Z_{j - 1} + A_{i_j}$ such that $A_{i_j} \cdot Z_{j - 1} > 0$ until we end with $Z = Z_l$. The integer $k = -Z^2$ is a useful invariant of the elliptic Gorenstein surface singularity (cf.\ \cite[Prop.\ 2.9]{C3-f}). Not all the exceptional curves introduced are $(-2)$-curves.
\end{Ex}

It follows from Proposition~\ref{prop: properties of rational gorenstein singularities} that the singularities we will now define are threefold rational Gorenstein singularities.

\begin{Def}
Let $X$ be a Gorenstein threefold and $P \in X$. If the general surface through $P$ has a rational singularity in $P$, i.e.\ a Du Val singularity, then we say that $P \in X$ is a \emph{compound Du Val singularity}, or $cDV$ singularity for short.
\end{Def}

In this case, in a suitable neighbourhood of $P$, we can see $X$ as a deformation of a Du Val singularity: the definition is equivalent (cf.\ \cite[Def.\ 2.1]{C3-f}) to ask that around $P$ the variety $X$ is locally analytically isomorphic to the hypersurface singularity in $\A^4$ given by
\[f(x, y, z) + t \cdot g(x, y, z, t) = 0,\]
where $f(x, y, z) = 0$ defines a Du Val singularity (see Table~\ref{table: du val singularities}). Observe also that while Du Val singularities are isolated, $cDV$ singularities can be isolated or not.

If instead the general surface through $P$ has an elliptic Gorenstein surface singularity, we can use the invariant $k$ introduced in Example~\ref{ex: surface elliptic gorenstein singularities} to classify them. We can summarize the classification in the following Proposition:

\begin{Prop}[{cf.\ \cite[Cor.\ 2.10]{C3-f}}]\label{prop: classification of rational gorenstein threefold singularities}
To a rational Gorenstein threefold singularity $P \in X$ one can attach an integer $k \geq 0$ such that
\begin{enumerate}
\item $k = 0$ if $P$ is a $cDV$ singularity;
\item $k \geq 1$ if the general surface $H$ through $P$ has an elliptic Gorenstein singularity whose fundamental cycle has self-intersection $-k$ (cf.\ Example~\ref{ex: surface elliptic gorenstein singularities}).
\end{enumerate}
In particular, if $k = 1$ then $P \in X$ is locally analytically isomorphic to the singularity in $\A^4$ defined by
\[y^2 = x^3 + f_1(s, t) x + f_2(s, t)\]
where $f_1$ is a sum of monomials of degree at least $4$ and $f_2$ is a sum of monomials of degree at least $6$.
\end{Prop}

\subsection{Canonical singularities}\label{sect: canonical singularities}
In this Section we will recall the definition of canonical singularities and some of their properties. Then we will focus on the link between canonical and $cDV$ singularities in the case of threefolds.

\begin{Def}[{cf.\ \cite[Def.\ 1.1]{YoungPerson}}]
Let $X$ be a (quasi-)projective normal variety. We say that $X$ has \emph{canonical singularities} if
\begin{enumerate}
\item\label{item: index of singularity} there exists an integer $r \geq 1$ such that the Weil divisor $r K_X$ is Cartier;
\item there exists a resolution $f: Y \longrightarrow X$ such that
\[r K_Y = f^*(r K_X) + \sum a_i E_i, \qquad a_i \geq 0,\]
where $r$ is as in point~\eqref{item: index of singularity} and $\{ E_i \}$ is the family of all the irreducible exceptional divisors of $f$.
\end{enumerate}
We say that $X$ has \emph{terminal singularities} if there exists a resolution of $X$ as above, such that $\{ E_i \} \neq \varnothing$ and $a_i > 0$ for every exceptional divisor $E_i$.
\end{Def}

For a point $P \in X$, the smallest $r$ for which point~\eqref{item: index of singularity} holds in a neighbourhood of $P$ is called the \emph{index} of $P$ in $X$. A divisor $E_i$ for which $a_i = 0$ is called a \emph{crepant divisor}, the others are called \emph{discrepant} (cf.\ with Definition~\ref{def: sing theory}).

\begin{Prop}[{cf.\ \cite[Cor.\ 1.11]{Pagoda}}]
On a threefold, the singularities of $cDV$ type are canonical (and of index $1$).
\end{Prop}
\begin{proof}
By definition, a variety is Cohen--Macaulay and canonical of index $1$ if and only if it has rational Gorenstein singularities. Then the result follows since $cDV$ singularities are threefold rational Gorenstein singularities (cf.\ Proposition~\ref{prop: properties of rational gorenstein singularities}).
\end{proof}

\begin{Prop}[{cf.\ \cite[Cor.\ 1.12]{Pagoda}}]\label{prop: small resolution implies cdv}
Let $X$ be a Gorenstein threefold with rational singularities. If $f: Y \longrightarrow X$ is a resolution of $X$ such that $\dim f^{-1}(P) \leq 1$ for all $P \in X$, then the singularities of $X$ are of $cDV$ type.
\end{Prop}
\begin{proof}
Since $X$ is Gorenstein with rational singularities, then $X$ has canonical singularities of index $1$. In \cite[Thm.\ 5.35]{Kollar-Mori} it is shown that if $X$ is a threefold with canonical singularities of index $1$ and $P \in X$ is a singular point, then the following are equivalent:
\begin{enumerate}
\item the general hypersurface section $P \in H \subseteq X$ is an elliptic singularity;
\item if $g: X' \longrightarrow X$ is any resolution of singularities then there is a crepant divisor $E \subseteq g^{-1}(P)$.
\end{enumerate}
Let $P \in X$ be a singular point. Since $\dim f^{-1}(P) \leq 1$, there is no divisor in $f^{-1}(P)$ and so the general hypersurface section through $P$ is not an elliptic singularity. Then the general hypersurface section through $P$ must have a rational (i.e.\ Du Val) singularity according to Proposition~\ref{prop: properties of rational gorenstein singularities}, which proves that $P$ is $cDV$.
\end{proof}

\subsubsection{Crepant resolutions}

In this Section we want to give some properties of crepant resolutions. In the following (in particular in Section~\ref{sect: crepant resolutions of weierstrass threefolds}) we will deal with such resolutions, so we collect here the basic results we will use.

\begin{Lemma}\label{lemma: positive discrepancy}
Let $X$ be a normal Gorenstein variety and $f: Y \longrightarrow X$ a resolution. If $D \subseteq \Exc f$ is an irreducible component such that $f(D) \cap (X \smallsetminus \Sing X) \neq \varnothing$, then $D$ is a discrepant divisor.
\end{Lemma}
\begin{proof}
For the generic $y \in D$, i.e.\ $y \in D \smallsetminus f^{-1}(\Sing X)$, we have that $f(y) \in X \smallsetminus \Sing X$. It follows from \cite[Prop.\ 5.8]{Kawamata} that then $D$ has codimension $1$ in $Y$. To prove that $D$ is discrepant, let $y \in D$ be generic as above and $x = f(y)$. Choose centred local local coordinates $(y_1, \ldots, y_n)$ and $(x_1, \ldots, x_n)$ around $y$ and $x$ respectively (where $n = \dim Y = \dim X$). So we have $(x_1, \ldots, x_n) = f(y_1, \ldots, y_n)$ and
\[f^*(dx_1 \wedge \ldots dx_n) = \det \left( J(f) \right) \cdot dy_1 \wedge \ldots dy_n\]
where $J(f)$ is the Jacobian matrix of $f$. Then the discrepancy of $D$ is the order of vanishing of $\det \left( J(f) \right)$ along $D$. Since $D \subseteq \Exc f$, $f$ is not an isomorphism around $x$, which implies that $\det \left( J(f) \right)$ vanishes and so that $D$ is discrepant.
\end{proof}

\begin{Cor}\label{cor: properties of crepant res}
Let $X$ be a normal Gorenstein variety and $f: Y \longrightarrow X$ a resolution. If $f$ is crepant, then:
\begin{enumerate}
\item $X$ has canonical singularities;
\item $\Exc f = f^{-1}(\Sing X)$.
\end{enumerate}
\end{Cor}
\begin{proof}
The first claim follows immediately from the definition of canonical singularity.\\
For the second, observe that we always have $f^{-1}(\Sing X) \subseteq \Exc f$. So we assume that $\Exc f \nsubseteq f^{-1}(\Sing X)$ and get a contradiction. Let $D$ be an irreducible component of $\Exc f$ which is not contained in $f^{-1}(\Sing X)$. By Lemma~\ref{lemma: positive discrepancy}, $D$ is a discrepant divisor and we are done.
\end{proof}

The previous Proposition shows that crepant resolutions are optimal in the sense that their exceptional locus is as small as possible. The following proposition shows that they are optimal also in the sense that they can not be factored through a divisorial contraction and another resolution of the singularities. A \emph{divisorial contraction} or the \emph{contraction of the divisor $D$} is a proper birational morphism $\varphi: X \longrightarrow X'$ such that $\Exc \varphi = D$ and $\varphi(D)$ has codimension at least $2$ in $X'$.

\begin{Prop}\label{prop: non contraction}
Let $X$ be a normal Gorenstein variety, and assume that we have a diagram
\[\xymatrix{Y \ar[rr]^f \ar[dr]_\varphi & & X\\
 & Y' \ar[ur]_{f'} & }\]
where $f$ is a resolution, $f'$ is a partial partial resolution and $\varphi$ is the contraction of a divisor $D$ in $Y$. Then
\begin{enumerate}
\item $D \subseteq \Exc f$;
\item if $f$ is crepant, then $\varphi$, $f'$ are crepant and $Y'$ has canonical singularities along $\varphi(D)$.
\end{enumerate}
\end{Prop}
\begin{proof}
Let $y \in D$ and define $y' = \varphi(y)$. Then $\dim \varphi^{-1}(y') > 1$ since $\varphi$ contracts $D$, and $f(\varphi^{-1}(y')) = f'(y')$ is a point. So $f$ is not an isomorphism near $f(y)$, i.e.\ $y \in \Exc f$, which proves that $D \subseteq \Exc f$.\\
Assume that $f$ is crepant, then we know that $X$ has canonical singularities by Corollary~\ref{cor: properties of crepant res}. So we can write
\[K_{Y'} = {f'}^* K_X + \sum a_i E_i, \qquad a_i \geq 0\]
where the $E_i$'s are the exceptional divisors introduced by $f'$. So
\[K_Y = \varphi^* K_{Y'} + a D = f^* K_W + \sum a_i \varphi^* E_i + a D.\]
We can write $\varphi^* E_i = \tilde{E}_i + b_i D$, where $\tilde{E}_i$ is the strict transform of $E_i$ in $Y$. But then we have
\[K_Y = f^* K_X + \sum a_i \tilde{E_i} + \left( a + \sum a_i b_i \right) D.\]
Since $f$ is crepant, we deduce that $a_i = 0$ for all $i$, hence that $a = 0$. This means that both $\varphi$ and $f'$ are crepant. By Lemma~\ref{lemma: positive discrepancy} we deduce that $Y'$ is not smooth and has canonical singularities along $\varphi(D)$.
\end{proof}

\subsubsection{Crepant resolutions of threefolds with $cDV$ singularities}\label{sect: resolving the du val locus}

In this Section we recall a result by M.\ Reid on the structure of crepant partial resolutions of threefolds with only $cDV$ singularities. It will be used in the proof of Theorem~\ref{thm: main theorem}.

\begin{Prop}[{cf.\ \cite[Thm.\ 1.14]{Pagoda}}]\label{prop: reid's main theorem}
Let $X$ be a threefold with $cDV$ sin\-gu\-la\-ri\-ties, and let $f: Y \longrightarrow X$ be a partial resolution. Then the following are equivalent:
\begin{enumerate}
\item $f$ is crepant;
\item $\dim f^{-1}(P) \leq 1$ for any $P \in X$, and $f$ if crepant above the generic point of any $1$-dimensional component of $\Sing X$;
\item for every $x$ in $X$ and every hypersurface $H$ through $x$ for which $x \in H$ is a Du Val singularity, $H' = f^{-1}(H)$ is normal and $f_{|_{H'}}: H' \longrightarrow H$ is crepant. Thus the minimal resolution of $x \in H$ factors through $H'$.
\end{enumerate}
\end{Prop}

\section{Elliptic fibrations}\label{section: elliptic fibration}

In this Section we want to recall the basic definitions and tools concerning the theory of elliptic fibrations. The theory of elliptic fibrations over a curve is well understood, see e.g.\ \cite{Kodaira1}, \cite{Kodaira}, \cite{Miranda}.

\begin{Def}\label{def: elliptic fibration}
We say that a morphism $\pi: X \longrightarrow B$ between normal varieties is an \emph{elliptic fibration} over $B$ if
\begin{enumerate}
\item $\pi$ is a proper morphism with connected fibres;
\item the generic fibre of $\pi$ is a smooth connected curve of genus $1$;
\item a section $\sigma: B \longrightarrow X$ of $\pi$ is given, i.e.\ $\sigma$ is a morphism such that $\pi \circ \sigma = \id_B$.
\end{enumerate}
\end{Def}

We denote $S = \sigma(B)$ the image of the section, and still call it a section.

\begin{Def}
Let $\pi: X \longrightarrow B$ be an elliptic fibration. If $\dim X_b = 1$ for all $b \in B$ (i.e.\ if every fibre is a curve), we say that $\pi$ is an \emph{equidimensional} elliptic fibration.
\end{Def}

A large class of examples of equidimensional elliptic fibrations is provided by Weierstrass fibrations.

\begin{Def}[{cf.\ \cite[Def.\ 1.1]{Nakayama}}]
A \emph{Weierstrass fibration} $p: W \longrightarrow B$ is a fibration described in a $\PP^2$-bundle over $B$ of the form $\PP(\cL^{\otimes 2} \oplus \cL^{\otimes 3} \oplus \cO_B)_{(x: y: x)}$ (for some line bundle $\cL$ on $B$) by a Weierstrass equation
\begin{equation}\label{formula: weierstrass equation}
y^2 z = x^3 + a_4 x z^2 + a_6 z^3, \qquad a_i \in H^0(B, \cL^{\otimes i}),
\end{equation}
such that $\Delta(p) = 4 a_4^3 + 27 a_6^2$ does not vanish identically. We will refer to a Weierstrass fibration, or to its total space, also as $W(\cL; a_4, a_6)$.
\end{Def}

Given two elliptic fibrations $\pi: X \longrightarrow B$ and $\pi': X' \longrightarrow B$, a \emph{morphism} (resp.\ a \emph{rational map}) of elliptic fibrations is a morphism $f: X \longrightarrow X'$ (resp.\ a rational map $f: X \dashrightarrow X'$) which is compatible with the fibrations, i.e.\ such that $\pi = \pi' \circ f$ (resp.\ $\pi = \pi' \circ f$ over the domain of $f$).

Exploiting the presence of the section there is a canonical way to associate to an elliptic fibration $\pi: X \longrightarrow B$ a birationally equivalent Weierstrass fibration: this is called the \emph{Weierstrass model} of the fibration.

\begin{Prop}[{cf.\ \cite[Thm.\ 2.1 and its proof]{Nakayama}}]\label{prop: existence weierstrass}
Let $\pi: X \longrightarrow B$ be a smooth elliptic fibration with section $s$ and let $S = s(B)$. Let $\cF = \left( \pi_* i_* \mathcal{N}_{S|X} \right)^{-1}$, where $i: S \hookrightarrow X$ is the inclusion. Then the canonical morphism $\pi^* \pi_* \cO_X(3 S) \longrightarrow \cO_X(3 S)$ is surjective, and defines a proper birational morphism $f: X \longrightarrow W$, where $W = W(\cF; a_4, a_6)$ is a (possibly singular) Weierstrass fibration.
\end{Prop}

The line bundle $\cF$ defined in Proposition~\ref{prop: existence weierstrass} is called the \emph{fundamental line bundle} of the fibration. It is possible to show (see \cite[$\S$II.4]{Miranda}) that
\[\pi_* \cO_X(3S) = \cF^{-2} \oplus \cF^{-3} \oplus \cO_B.\]
and that $\cF$ is intrinsic of the fibration and do not depend on the choice of the section, since
\[\cF \simeq \left( R^1 \pi_* \cO_X \right)^{-1}.\]

\begin{Rmk}
We can describe the morphism $f: X \longrightarrow W$ as the contraction, in the reducible fibres, of all the irreducible components which do not meet the section (cf.\ \cite[Rmk.\ 2.8]{Nakayama}).
\end{Rmk}

\begin{Rmk}
Let $\pi: X \longrightarrow B$ be a smooth elliptic fibration. Then the morphism $f: X \longrightarrow W$ to its Weierstrass model is a resolution of $W$.
\end{Rmk}

\begin{Rmk}
As any hypersurface in a smooth ambient space, the Weierstrass model $W$ of an elliptic fibration is a Gorenstein variety (cf.\ Example~\ref{ex: gorenstein varieties}). We can find this fact also in \cite[p.\ 409]{Nakayama}, where there is also a formula for the dualizing sheaf of $p: W \longrightarrow B$:
\[\omega_W = p^* \left( \omega_B \otimes \cF \right).\]
\end{Rmk}

\begin{Def}
The \emph{discriminant locus} of an elliptic fibration $\pi: X \longrightarrow B$ is the subset of $B$ which parametrizes the singular fibres of $\pi$:
\[\Delta(\pi) = \left\{ b \in B \st X_b \text{ is a singular curve} \right\}.\]
\end{Def}

\begin{Rmk}
Let $p: W \longrightarrow B$ be a Weierstrass fibration. Then we can describe its discriminant by an equation, which is the usual discriminant of the Weierstrass cubic polynomial (\ref{formula: weierstrass equation}):
\[\Delta(p): 4 a_4^3 + 27 a_6^2 = 0.\]
So $\Delta(p)$ is not only a subset of $B$, but also a subscheme.
\end{Rmk}

\begin{Rmk}\label{rmk: inclusion discriminants}
Let $\pi: X \longrightarrow B$ be an elliptic fibration, with Weierstrass model $p: W \longrightarrow B$. If $X_b$ is a smooth elliptic curve then $W_b$ is smooth, in fact in this case $W_b$ is the Weierstrass model of $X_b$. This implies that $B \smallsetminus \Delta(\pi) \subseteq B \smallsetminus \Delta(p)$, or equivalently that $\Delta(p) \subseteq \Delta(\pi)$.
\end{Rmk}

\begin{Ex}\label{ex: different discriminants}
Let $E = \{ v^2 w = u^3 + \alpha u w^2 + \beta w^3 \} \subseteq \PP^2_{(u: v: w)}$ be an elliptic curve in Weierstrass form, with zero $O = (0: 1: 0)$, and let $B$ be a smooth surface. Define $W = B \times E \subseteq B \times \PP^2$, the constant fibration with structure map $p: W \longrightarrow B$ and section $S = B \times \{ O \}$. Choose a smooth curve $C \subseteq B$ and a point $Q \in E \smallsetminus \{ O \}$, and let $f: X \longrightarrow W$ be the blow up of $W$ in $C \times \{ Q \}$. Then $\pi = p \circ f$ defines on $X$ an equidimensional elliptic fibration over $B$, whose section is the strict transform of $S$. The fibre over $P \in C$ is singular, as it is reducible: its irreducible components are the strict transform of the curve $W_P$ and the rational curve introduced by the blow up. The Weierstrass model of $X$ is $W$, and in this example we have $\Delta(p) \subsetneqq \Delta(\pi)$: in fact $\Delta(p)$ is empty while $\Delta(\pi)$ is the curve $C$.
\end{Ex}

As pointed out in Example~\ref{ex: different discriminants}, we can have elliptic fibrations whose discriminant is different from the discriminant of their Weierstrass model. However we can avoid this effect if our original fibration is sufficiently close to its Weierstrass model.

\begin{Prop}\label{prop: same discriminant}
Let $\pi: X \longrightarrow B$ be a smooth elliptic fibration and $f: X \longrightarrow W$ be the morphism to its Weierstrass model $p: W \longrightarrow B$. If $\Exc f = f^{-1}(\Sing W)$, then $\Delta(\pi) = \Delta(p)$.
\end{Prop}
\begin{proof}
In view of Remark~\ref{rmk: inclusion discriminants} we only need to show that $\Delta(\pi) \subseteq \Delta(p)$, or equivalently that for any $b \in B$ such that $W_b$ is smooth, then $X_b$ is smooth. But if $W_b$ is a smooth fibre, then $W_b \subseteq W \smallsetminus \Sing W \simeq X \smallsetminus \Exc f$ and so $f$ induces an isomorphism $X_b \simeq W_b$.
\end{proof}

\section{Crepant resolutions of Weierstrass threefolds}\label{sect: crepant resolutions of weierstrass threefolds}

In this Section we investigate the structure of the non-Kodaira fibres in a smooth equidimensional elliptic threefold which is a crepant resolution of its Weierstrass model. Before going deep into the subject we want to give some reasons why this class of fibrations is interesting to study, so we take a little detour to the theory of elliptic surfaces.

Let $\pi: X \longrightarrow B$ be a smooth elliptic surface. The following well known definition is due to Kodaira.

\begin{Def}[{cf.\ \cite[p.\ 564]{Kodaira}}]
Let $\pi: X \longrightarrow B$ be a smooth elliptic surface. We say that it is \emph{minimal elliptic} if no fibre of $\pi$ contains a $(-1)$-curve of $X$.
\end{Def}

Being minimal for an elliptic surface is crucial to obtain a classification of the possible singular fibres that can occur: without that assumption, we can blow up points on the surface and obtain in this way many and many different singular fibres. It is then important to restrict ourselves to a suitable class of elliptic fibrations if we have in mind classificatory purposes. In the case of surfaces, \cite[Thm. 6.2]{Kodaira} gives a complete classification of all the possible singular fibres of a smooth minimal elliptic surface, which we call \emph{singular fibres of Kodaira type} or simply \emph{Kodaira fibres}.

The following Proposition is important since it provides a clear link between minimal elliptic surfaces and their Weierstrass models.

\begin{Prop}[{cf.\ \cite[Prop.\ III.3.2]{Miranda}}]
Let $p: W = W(\cL; a_4, a_6) \longrightarrow B$ be a Weierstrass fibration over the curve $B$. Then the following are equivalent:
\begin{enumerate}
\item $W$ is the Weierstrass model of a smooth minimal elliptic surface;
\item $W$ has only Du Val singularities;
\item there is no point $b \in B$ such that
\[\left\{ \begin{array}{l}
\mult_b a_4 \geq 4\\
\mult_b a_6 \geq 6.
\end{array} \right.\]
\end{enumerate}
\end{Prop}

In the hypothesis of the above Proposition, if we assume that $W$ is the Weierstrass model of some smooth elliptic surface $\pi: X \longrightarrow B$, then that surface is in fact \emph{the unique} smooth minimal elliptic surface having $W$ as Weierstrass model. By the Proposition, such $W$ has Du Val singularities, and $\pi$ is defined by the minimal resolution $f: X \longrightarrow W$ of $W$, which is crepant (see Example~\ref{ex: surface rational gorenstein singularities}). Moreover, any smooth minimal elliptic surface can be described in this way.

\begin{Rmk}
It follows from this description that we can define equivalently a smooth minimal elliptic surface as an elliptic surface which is the crepant resolution of a Weierstrass fibration having only Du Val singularities.
\end{Rmk}

We adopt this as a starting point for our analysis of threefolds.

Let $\pi: X \longrightarrow B$ be a smooth elliptic threefold. Then the morphism to the Weierstrass model $f: X \longrightarrow W$ is a resolution of $W$. We want to use such morphism as a ``measure'' of the complexity of the original fibration: for the simplest one, which should be closer to their Weierstrass model, we can make some attempt to classify their non-Kodaira fibres.\\
Here is another reason why we are interested in crepant resolutions and in their properties. One way to produce smooth elliptic threefolds is to start with a singular Weierstrass fibration and then desingularize it. In this process one generally tries to keep the desingularization as similar as possible as the Weierstrass fibration he started with: it is usual to look for crepant resolutions. This is important especially in $F$-theory, where one starts with a singular Calabi--Yau Weierstrass fibration, and look for a resolution whose total space is still a Calabi--Yau variety.

As we saw in Example~\ref{ex: different discriminants}, an elliptic fibration can be very different from its Weierstrass model simply because the first can have a discriminant locus which is more complicated than the discriminant locus of the Weierstrass model. The following Lemma (which holds in any dimension, not only in the case of threefolds) shows that an elliptic fibration which is a crepant resolution of its Weierstrass model is not too far from it in the sense that they have the same discriminant.

\begin{Lemma}\label{lemma: crepant implies same discriminant}
Let $\pi: X \longrightarrow B$ be a smooth elliptic fibration, with Weierstrass model $p: W \longrightarrow B$. If the morphism to the Weierstrass model $f: X \longrightarrow W$ is crepant, then $\Delta(\pi) = \Delta(p)$.
\end{Lemma}
\begin{proof}
The result easily follows from Corollary~\ref{cor: properties of crepant res} and Proposition~\ref{prop: same discriminant}.
\end{proof}

Assume that $W = W(\cL; a_4, a_6)$ is a Weierstrass fibration with $W$ a threefold. If we have some control on the coefficients $a_4$ and $a_6$, we can ensure that $W$ has rational Gorenstein singuarities.

\begin{Lemma}[{cf.\ \cite[Lemma 3.6]{Nakayama}}]\label{lemma: mult implies rat gor}
Let $p: W = W(\cL; a_4, a_6) \longrightarrow B$ be a Weierstrass fibration over a smooth surface $B$. If $\mult_b a_4 \leq 3$ or $\mult_b a_6 \leq 5$ for any $b \in B$, then $W$ has rational Gorenstein singuarities.
\end{Lemma}

\begin{Lemma}\label{lemma: cdv implies mult}
Let $p: W = W(\cL; a_4, a_6) \longrightarrow B$ be a Weierstrass fibration over a smooth surface $B$. If $W$ has only $cDV$ singularities, then $\mult_b a_4 \leq 3$ or $\mult_b a_6 \leq 5$ for any $b \in B$.
\end{Lemma}
\begin{proof}
First of all, observe that $W$ has rational Gorenstein singularities. Assume that there is a point $b \in B$ such that
\[\left\{ \begin{array}{l}
\mult_b a_4 \geq 4\\
\mult_b a_6 \geq 6.
\end{array} \right.\]
So we can find local coordinates $(s, t)$ around $b$ such that a (local) equation for $W$ is given by
\[y^2 = x^3 + a_4(s, t) x + a_6(s, t)\]
where $a_4$ and $a_6$ are sums of monomials of degree at least $4$ and $6$ respectively. By Proposition~\ref{prop: classification of rational gorenstein threefold singularities}, this is the canonical form of a rational Gorenstein threefold singularity with invariant $k = 1$. Since these singularities are not $cDV$, we get a contradiction.
\end{proof}

We begin to study the link between Weierstrass fibrations and equidimensional elliptic threefolds. If $\pi: X \longrightarrow B$ is a smooth equidimensional elliptic fibration, then we have a full control on the singularities of its Weierstrass model $W$.

\begin{Lemma}\label{lemma: better weierstrass}
Let $p: W = W(\cL; a_4, a_6) \longrightarrow B$ be a Weierstrass fibration over a smooth surface $B$. Assume that $W$ has a resolution $f: X \longrightarrow W$ such that $\dim f^{-1}(w) \leq 1$ for every $w \in W$. Then the following conditions are equivalent:
\begin{enumerate}
\item $W$ has $cDV$ singularities.
\item $\mult_b a_4 \leq 3$ or $\mult_b a_6 \leq 5$ for any $b \in B$.
\item $W$ has rational Gorenstein singularities.
\end{enumerate}
\end{Lemma}
\begin{proof}
$(1) \Longrightarrow (2)$ is Lemma~\ref{lemma: cdv implies mult}.\\
$(2) \Longrightarrow (3)$ is Lemma~\ref{lemma: mult implies rat gor}.\\
$(3) \Longrightarrow (1)$ follows from Proposition~\ref{prop: small resolution implies cdv}.
\end{proof}

\begin{Rmk}\label{rmk: canonical iff cdv}
In particular, we can apply Lemma~\ref{lemma: better weierstrass} when $W$ is the Weierstrass model of a smooth equidimensional elliptic threefold, and in this case the class of canonical singularities coincides with the one of $cDV$ singularities.
\end{Rmk}

We are now ready to prove our main Theorem.

\begin{Thm}\label{thm: main theorem}
Let $\pi: X \longrightarrow B$ be a smooth equidimensional elliptic threefold, with morphism to the Weierstrass model $f: X \longrightarrow W$. If $f$ is a crepant resolution of $W$, then
\begin{enumerate}
\item $W$ has $cDV$ singularities;
\item\label{item: no contraction} $f$ can not be factored as
\[\xymatrix{X \ar[dr]_\varphi \ar[rr]^f & & W\\
 & X' \ar[ur]_{f'} & }\]
where $\varphi$ is the contraction of some divisor $D \subseteq X$ and $X'$ is smooth;
\item if $X_b$ is a non-Kodaira fibre, then $b \in \Sing (\Delta(\pi)_\red)$;
\item if $X_b$ is a non-Kodaira fibre, then $X_b$ is a contraction of a Kodaira fibre, whose type can be determined by Tate's algorithm.
\end{enumerate}
\end{Thm}
\begin{proof}
Since $f$ is crepant, $W$ has canonical singularities by Corollary~\ref{cor: properties of crepant res}. These singularities are $cDV$ by Remark~\ref{rmk: canonical iff cdv}.\\
The second assertion follows directly from Proposition~\ref{prop: non contraction}.\\
We will now prove the last two statements. Observe that by Lemma~\ref{lemma: crepant implies same discriminant}, $\pi$ and $p$ have the same discriminant, so we consider $\Delta(\pi) = \Delta(p)$ also as subschemes of $B$. Let $Q \in \Delta(\pi)$ be generic, meaning that $Q$ is a smooth point for the reduced discriminant $(\Delta(\pi))_\red$, and let $C$ be a (local) curve through $Q$ which is smooth at $Q$. For the generic $C$ (i.e.\ meeting $(\Delta(\pi))_\red$ transversally) we have that
\begin{enumerate}
\item $X_{|_C} = \pi^{-1}(C)$ is a smooth elliptic surface;
\item $W_{|_C} = p^{-1}(C)$ is a Weierstrass surface with Du Val singularities;
\item $f_{|_C}: X_{|_C} \longrightarrow W_{|_C}$ is the morphism to the Weierstrass model.
\end{enumerate}
By Proposition~\ref{prop: reid's main theorem}, $f_{|_C}$ factors the minimal resolution of $W_{|_C}$, and since $X_{|_C}$ is smooth we deduce that $f_{|_C}$ coincides with the minimal resolution of $W_{|_C}$. This means that $X_{|_C}$ is a smooth minimal elliptic surface. Hence, over the smooth points of $(\Delta(\pi))_\red$ the corresponding fibre is a Kodaira fibre. So the non-Kodaira fibres are located over the singular points of $(\Delta(\pi))_\red$.\\
Let $b \in \Sing ((\Delta(\pi))_\red)$ be a point such that $X_b$ is of non-Kodaira type, and consider a curve $C$ through $b$ such that $\mult_b {a_4}_{|_C} \leq 3$ or $\mult_b {a_6}_{|_C} \leq 5$ (in fact the generic $C$ will work). By Proposition~\ref{prop: reid's main theorem}, $X_{|_C} \longrightarrow W_{|_C}$ factors the minimal resolution of $W_{|_C}$ and so $X_b$ is the contraction of a Kodaira fibre, whose type can be predicted by Tate's algorithm.
\end{proof}

The second statement in Theorem~\ref{thm: main theorem} is interesting, since it tells us that the smooth equidimensional elliptic fibrations which are crepant resolutions of their Weierstrass models are close to the Weierstrass model also in another weaker sense.\\
Assume that $f: X \longrightarrow X'$ is a birational morphism of smooth projective varieties. Then its exceptional locus is of pure codimension $1$ (cf.\ \cite[Prop.\ 5.8]{Kawamata} or \cite[$\S$1.40]{Debarre}). So, if we have a smooth elliptic fibration $\pi: X \longrightarrow B$ then part~\eqref{item: no contraction} of Theorem~\ref{thm: main theorem} says that we can not simplify our $X$ by means of some divisorial contraction to get a new smooth elliptic fibration $\pi': X' \longrightarrow B$ (where $\pi' = p \circ f'$) with the same Weierstrass model $W$ as $X$.

\section{Examples and non-examples}\label{sect: examples}

In this Section we give some examples of elliptic threefolds with non-Kodaira fibres, and show that they are contracions of Kodaira fibres, as stated in Theorem~\ref{thm: main theorem}.\\
In Example~\ref{ex: CCvG} and Example~\ref{ex: non-moving singularities} we will construct explicitly some of these fibres, which are the contraction of fibres of Kodaira type $I_0^*$ and $I_1^*$ respectively, and in both cases it is possible to check that the morphism to the Weierstrass model is crepant. In Example~\ref{ex: a non-example} we will show that some configurations of curves can not appear as fibre of a smooth equidimensional elliptic threefold for which the morphism to the Weierstrass model is crepant. Example~\ref{ex: resolution of non-cdv} is devoted to the construction of a crepant resolution of a Weierstrass fibration which does not satisfy the condition
\begin{equation}\label{eq: no points with high multiplicity}
\mult_b a_4 \leq 3 \text{ or } \mult_b a_6 \leq 5 \text{ for all } b \in B.
\end{equation}
As we will see, the resolution will not be an equidimensional elliptic fibration.

\begin{Ex}
One of the first papers where non-Kodaira fibres appeared is \cite{SME3-f}, that article is in fact devoted to an explicit desingularization of certain Weierstrass threefolds. The method described there allows one to blow up the base, in order to deal with a simpler fibration (e.g.\ one can blow up the base surface until the discriminant curve has only normal crossing). With this freedom, the classification of the non-Kodaira fibres in the paper is compete, and in \cite[$\S$14]{SME3-f} there is the observation that all the non-Kodaira fibres found in this way are contracion of the right Kodaira fibre, in agreement with Theorem~\ref{thm: main theorem}.
\end{Ex}

\begin{Ex}[{cf.\ \cite[$\S$2]{CCvG}}]\label{ex: CCvG}
Let $Z$ be the projective bundle $\PP(\cO_{\PP^2}(3) \oplus \cO_{\PP^2} \oplus \cO_{\PP^2})$ over $\PP^2$ and consider the hypersurface
\begin{equation}\label{equation: CCvG}
X: x^2 y + f y^3 + g (y^2 z + z^3) = 0,
\end{equation}
where $(x: y: z)$ are coordinates in the fibres and $f$, $g$ are homogeneous sextic polynomials defining smooth plane sextics intersecting transversally in $36$ distinct points. Then $X$ is a smooth variety, and restricting to $X$ the bundle projection we get an elliptic fibration over $\PP^2$, with section $\sigma$ given by $P \longmapsto (1: 0: 0) \in X_P$. The discriminant of the family is the curve
\[\Delta: g^4 (27 f^2 + 4 g^2) = g^4 (2 g + 3 \sqrt{-3} f) (2 g - 3 \sqrt{-3} f) = 0;\]
over the curve $g = 0$ the fibre is of Kodaira type $IV$, while over each of the curves $2 g \pm 3 \sqrt{-3} f = 0$ we have nodal cubics. Over the $36$ points where $f = g = 0$, the fibre has equation $x^2 y = 0$ and so consists of two concurring lines, one of which with multiplicity $2$ and so this fibre is not of Kodaira type. A picture of the singular fibres of this fibration is in Figure~\ref{figure: fibres of CCvG}.

\begin{center}
\begin{figure}[H]
\includegraphics[width = 0.8\textwidth]{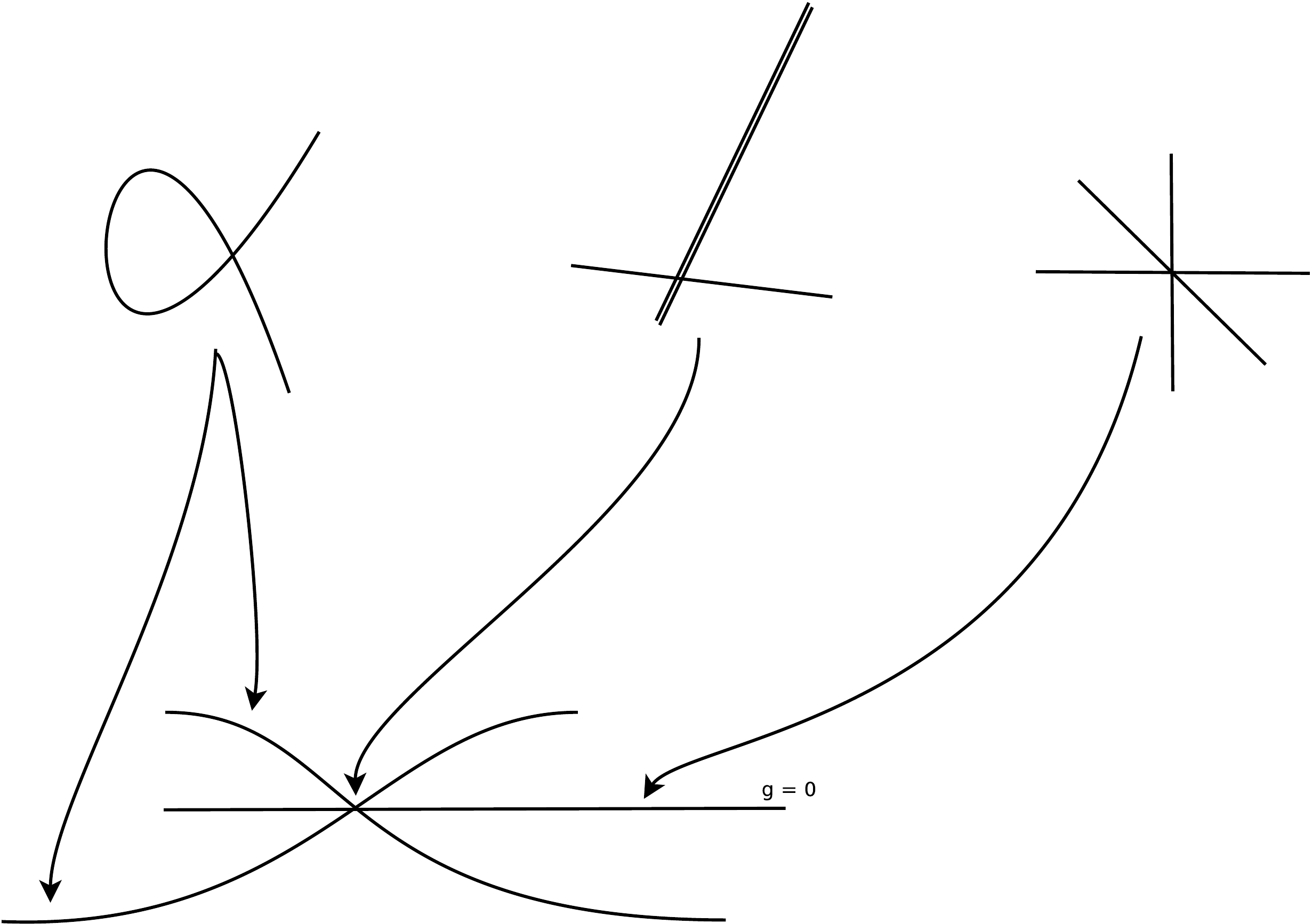}
\caption{The singular fibres of the fibration (\ref{equation: CCvG}).}
\label{figure: fibres of CCvG}
\end{figure}
\end{center}

Around each point where $f = g = 0$ there is a suitable neighbourhood where $f$ and $g$ give local centred coordinates. In such a neighbourhood, we choose a generic line through the origin, say $g = \lambda f$, and define $X_\lambda$ to be the restriction of the fibration to this line. By Tate's algorithm, the fibre we expect over $f = 0$ is of Kodaira type $I_0^*$: the fibre of the threefold is a contraction of such fibre, in agreement with Theorem~\ref{thm: main theorem}. In Figure~\ref{figure: contraction of I_0^*} it's possible to see which components have been contracted.

\begin{center}
\begin{figure}[H]
\includegraphics[width = 0.8\textwidth]{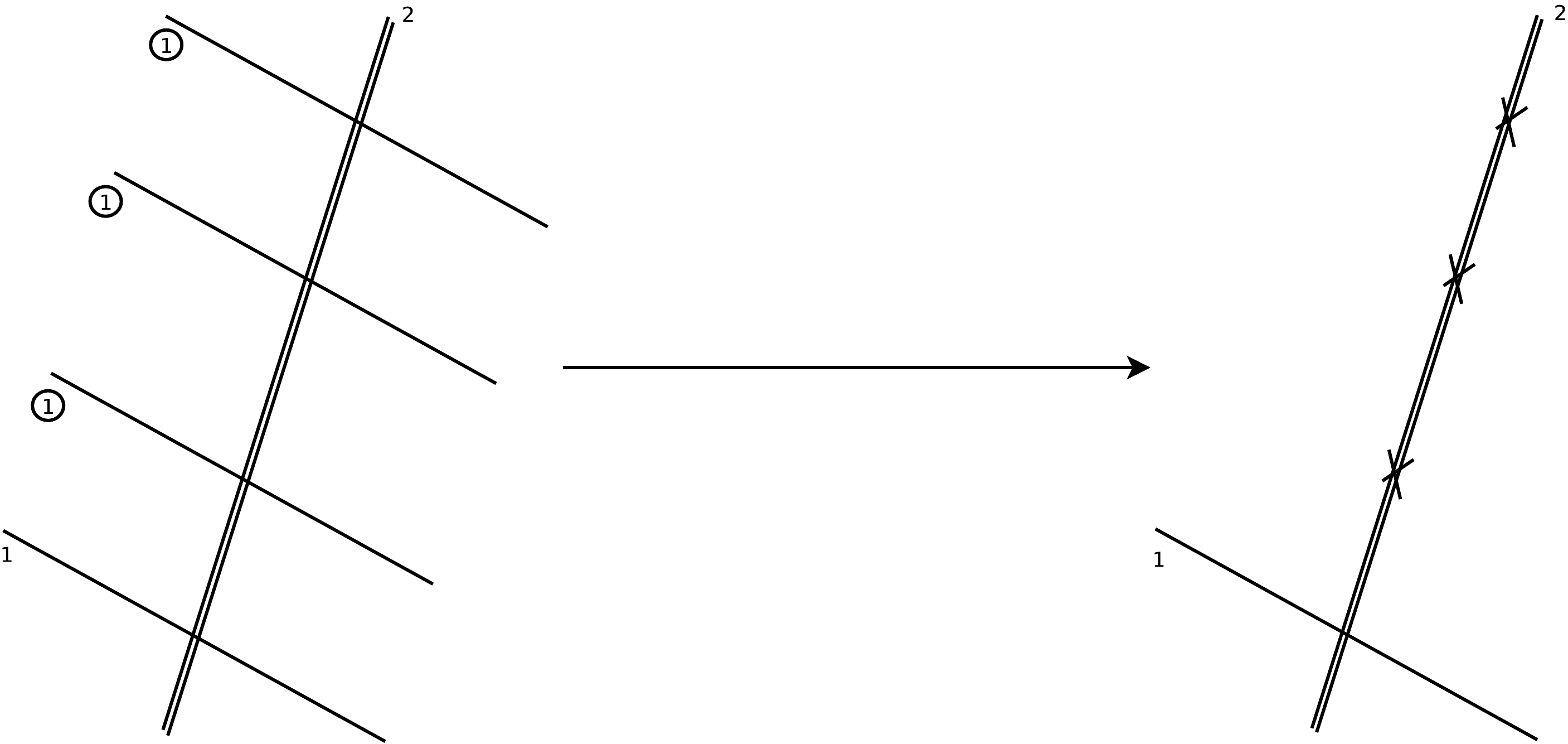}
\caption{On the left a $I_0^*$ fibre, and on the right the non-Kodaira fibre of fibration (\ref{equation: CCvG}). The circled components are blown down: they are the multiplicity $1$ components which do not meet the section.}
\label{figure: contraction of I_0^*}
\end{figure}
\end{center}

\begin{Rmk}\label{rmk: moving singularities}
Observe that for different values of the parameter $\lambda$, the three singular points in the central fibre of $X_\lambda$ change coordinates. As a consequence these points, which are singular only for some values of $\lambda$, are smooth points for the threefold.
\end{Rmk}
\end{Ex}

In Example~\ref{ex: CCvG}, a point of the multiplicity $2$ component of the non-Kodaira fibres is a smooth point for the elliptic surface obtained by restriction to the \emph{generic} curve through the points where $f = g = 0$. As observed in Remark~\ref{rmk: moving singularities}, this is enough to ensure that the threefold is smooth at those points. The next example will show that this condition is not necessary.

\begin{Ex}\label{ex: non-moving singularities}
Consider in $\PP(\cO_{\PP^2}(4) \oplus \cO_{\PP^2}(6) \oplus \cO_{\PP^2})$ the fibration in Weierstrass form
\begin{equation}\label{equation: weierstrass of fibration with I_1^*}
W: y^2 z = x^3 - \frac{1}{3} t^2 x z^2 + \left( s^4 + \frac{2}{27} t^3 \right) z^3,
\end{equation}
where $s = 0$ and $t = 0$ define a smooth cubic and a smooth quartic respectively, with transverse intersection in $12$ distinct points. The discriminant of this fibration is
\[\Delta = s^4 (27 s^4 + 4 t^3)\]
and so by Tate's algorithm we expect that a resolution $\pi: X \longrightarrow \PP^2$ of $W$ has
\begin{enumerate}
\item $I_4$ fibres over the line $s = 0$. 
\item $I_1$ fibres over $27 s^4 + 4 t^3 = 0$.
\end{enumerate}
The curve $s = x = y = 0$ is singular for the threefold $W$, hence we blow it up. The effect is that over $s = 0$ instead of nodal cubics now we have triangles, with one of the vertices which is still singular for the whole variety. After a second blow up of this curve of singular points we have a smooth threefold with $I_4$ fibres over $s = 0$, as expected. Observe that the exceptional divisors we introduce in this way are crepant. The fibre over the points where $s = 0$ and $t = 0$ meet are of non-Kodaira type, and we have a picture of this fibres in Figure~\ref{figure: fibres of fibration with I_1^*}.

\begin{center}
\begin{figure}[H]
\includegraphics[width = 0.8\textwidth]{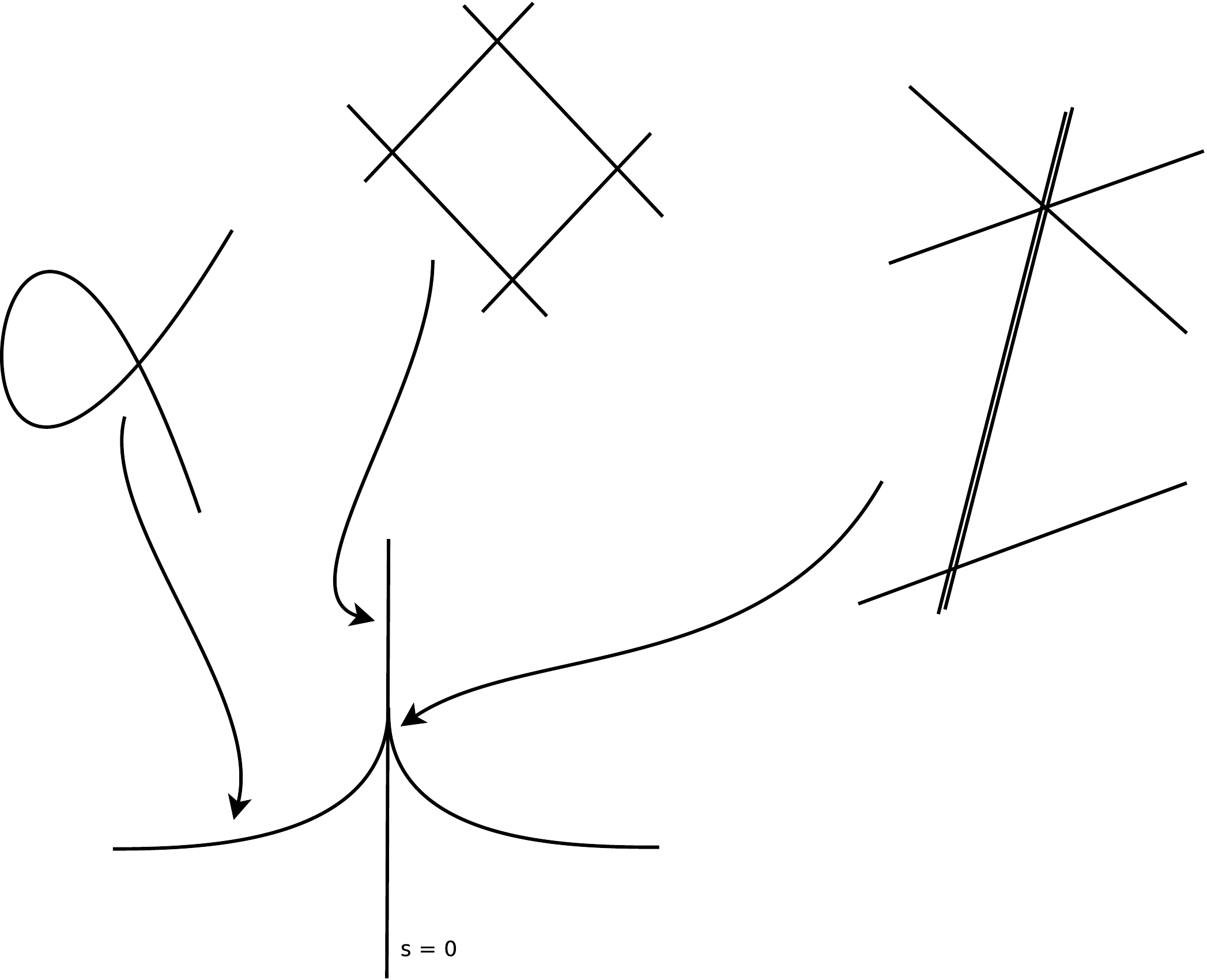}
\caption{The singular fibres of the resolution of (\ref{equation: weierstrass of fibration with I_1^*}).}
\label{figure: fibres of fibration with I_1^*}
\end{figure}
\end{center}

We want now to give a local description of the fibration around a point over which we have the new fibres, and in a suitable neighbourhood of such points local coordinates are given by $s$ and $t$. Let $X_\lambda$ be the restriction of the fibration to the generic line through the origin, i.e.\ $t = \lambda s$. By Tate's algorithm on this elliptic surface $X_\lambda$, we should have over the origin a fibre of Kodaira type $I_1^*$: what we see is not the whole fibre but a contraction of it (Figure~\ref{figure: contracion of I_1^*}).

\begin{center}
\begin{figure}[H]
\includegraphics[width = 0.8\textwidth]{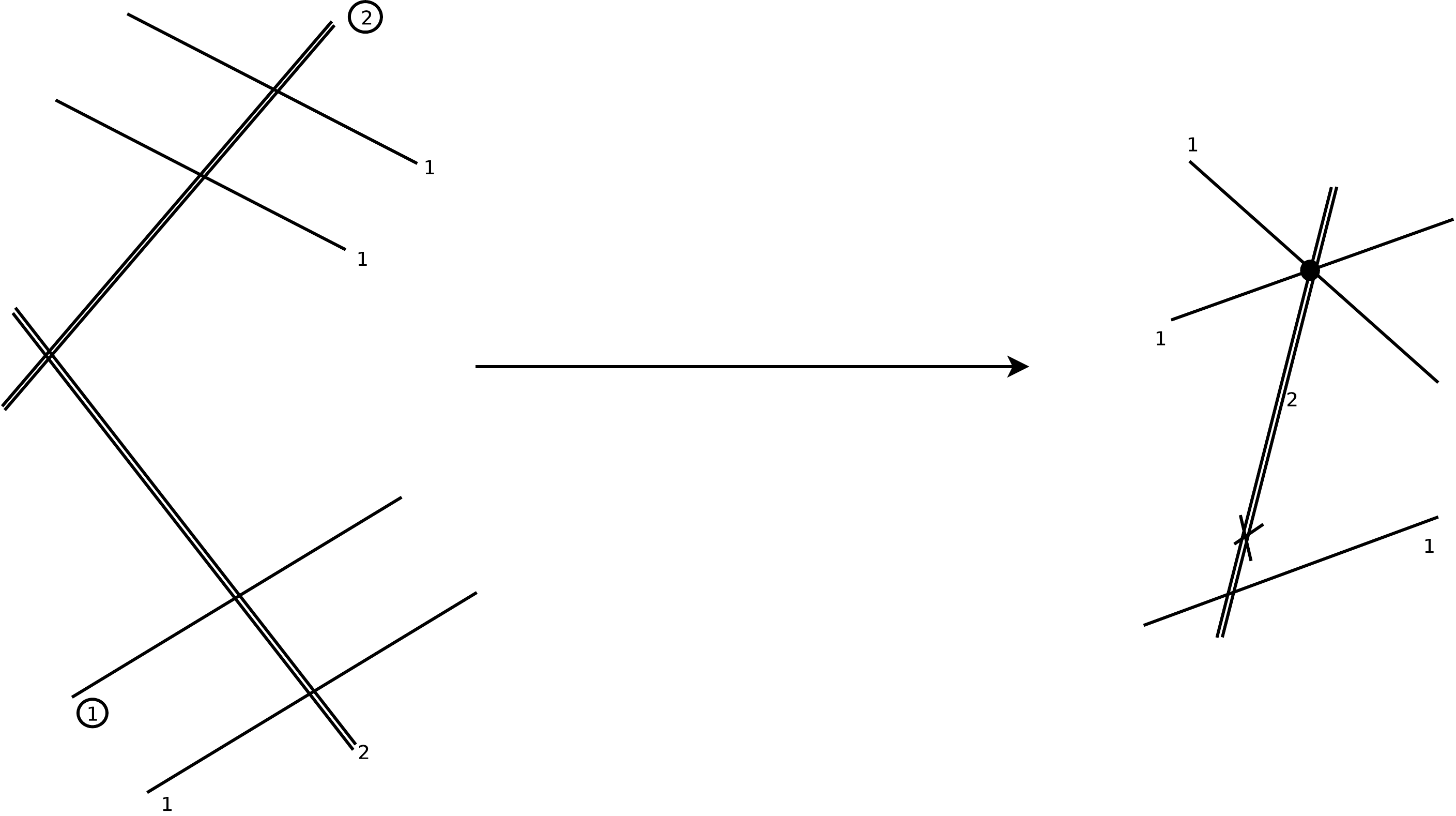}
\caption{On the left a $I_1^*$ fibre, and on the right the non-Kodaira fibre on the resolution of (\ref{equation: weierstrass of fibration with I_1^*}). The circled components are contracted.}
\label{figure: contracion of I_1^*}
\end{figure}
\end{center}

This example is interesting also for the following reason. The surface $X_\lambda$ is singular over $s = 0$, since we do not have a Kodaira fibre, and so there are singular points whose coordinates are expected to depend on $\lambda$. In fact there are always two singular points: one of them is on the multiplicity two component of the non-Kodaira fibre and its coordinates depend on $\lambda$, while the second is the point of the fibre where the multiplicity $2$ component of the $I_1^*$ fibre was blown down. This last has coordinates independent of $\lambda$, but nevertheless the threefold is smooth at this point.
\end{Ex}

\begin{Ex}
The fibration described in Example~\ref{ex: non-moving singularities} fits in a more general class of examples, which also show that~\eqref{eq: no points with high multiplicity} gives a necessary condition for a Weierstrass fibration to be the Weierstrass model of some smooth equidimensional elliptic threefold (according to Lemma~\ref{lemma: better weierstrass}), but that this condition is far from being sufficient.\\
Consider the local elliptic fibration
\[y^2 z = x^3 + t^m x^2 z + s^n z^3, \qquad m, n \geq 1,\]
in $U \times \PP^2$, where $U$ is a neighbourhood of the origin in $\A^2_{(s, t)}$, while $(x: y: z)$ are coordinates in the fibre $\PP^2$.\\
We can put the equation in Weierstrass form, obtaining the standard equation $y^2 z = x^3 + a_4 x z^2 + a_6 z^3$ with
\[\begin{array}{l}
a_4 = -\frac{1}{3} t^{2m},\\
a_6 = s^n + \frac{2}{27} t^{3m};
\end{array}\]
and so the discriminant locus of the family is
\[\Delta: s^n (27 s^n + 4 t^{3m}) = 0.\]
By Lemma~\ref{lemma: better weierstrass}, not for every pair $(m, n)$ the corresponding variety can be the Weierstrass model of an equidimensional elliptic fibration. We have to exclude all the cases with
\[\left\{ \begin{array}{l}
2m \geq 4\\
\min (n, 3m) \geq 6
\end{array} \right. \longrightarrow \left\{ \begin{array}{l}
m \geq 2\\
n \geq 6.
\end{array} \right.\]

\begin{Rmk}
The fibration described in Example~\ref{ex: non-moving singularities} corresponds to $(m, n) = (1, 4)$.
\end{Rmk}

\begin{Rmk}
All the cases with $(m, n) = (m, 1)$ are smooth, and the origin is a point of the discriminant where the two components meet with arbitrarily large multiplicity.
\end{Rmk}

\begin{Rmk}
Consider the cases with $(m, n) = (m, 3)$ and $m \geq 2$. After one blow up in the singular locus, we obtain a variety which is still singular in three points, where each singularity is locally isomorphic to
\[y^2 = x^2 + s^2 + t^m.\]
As pointed out in \cite[Cor.\ 1.16]{Pagoda}, these singularities are $cDV$ and admit a small resolution if and only if $m$ is even. This means that not every Weierstrass fibration satisfying the condition~\eqref{eq: no points with high multiplicity} is actually the Weierstrass model of a smooth equidimensional elliptic fibration.
\end{Rmk}
\end{Ex}

\begin{Ex}\label{ex: a non-example}
We now want to show that not every singular curve of arithmetic genus $1$ can be a fibre in a smooth equidimensional elliptic threefold.\\
Let $F$ be the curve defined in $\PP^3$ by the equations
\[\left\{ \begin{array}{l}
x_0 x_3 = 0\\
x_1 x_3 = 0\\
x_1^2 x_2 = x_0^3.
\end{array} \right.\]
This curve has two irreducible components: a plane cuspidal cubic and a non-coplanar line passing through the cusp (Figure~\ref{figure: cusp with a line}).

\begin{center}
\begin{figure}[H]
\includegraphics[width = 0.4\textwidth]{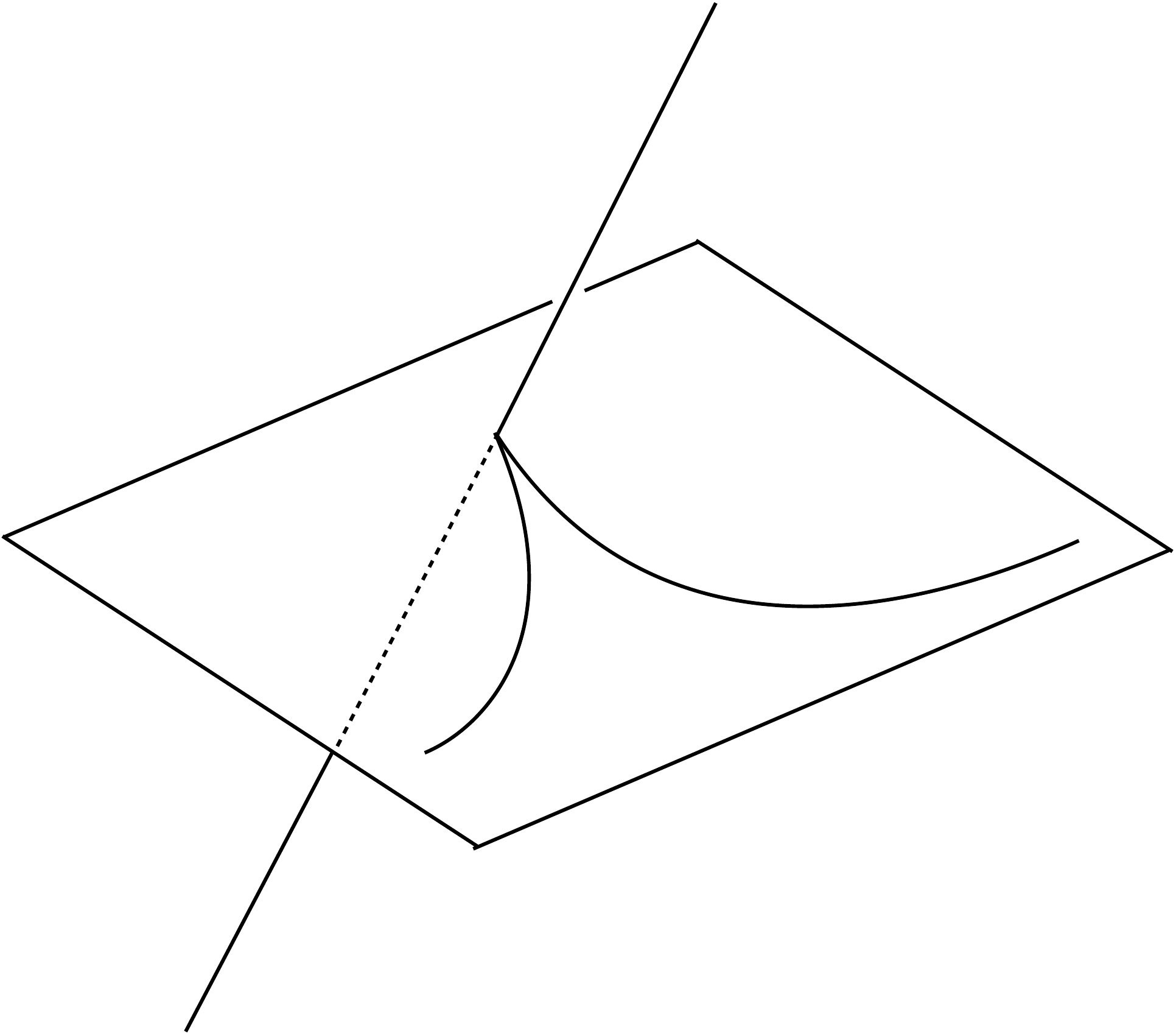}
\caption{The fibre $F$.}
\label{figure: cusp with a line}
\end{figure}
\end{center}

Such a curve is not of Kodaira type and so there is no elliptic surface having it as fibre. We could deduce this fact also from the following observation: the tangent space to the singular point of $F$ has dimension $3$, so we need at least an elliptic threefold to have a smooth total space having $F$ among its fibres. However, we can not obtain $F$ as a contraction of a Kodaira fibre, and so no smooth equidimensional elliptic threefold for which the morphism to the Weierstrass model is crepant can have such a singular fibre.
\end{Ex}

\begin{Ex}\label{ex: resolution of non-cdv}
In this Example we want to show how we can explicitly resolve the Weierstrass fibration defined by the equation
\[W: y^2 = x^3 + s^4 x + t^6,\]
which defines a threefold in $\A^4_{(s, t, x, y)}$ with an isolated singularity at the origin. This singularity is rational Gorenstein, but not of $cDV$-type by Proposition~\ref{prop: classification of rational gorenstein threefold singularities}. The resolution we are going to show is described also in \cite[Thm.\ 2.11]{C3-f} and introduces a crepant exceptional divisor over the singular point. By \cite[Thm.\ 5.35]{Kollar-Mori}, this singularity has no small resolutions.

Consider in $\A^4_{(s, t, x, y)} \times \PP^{(1, 1, 2, 3)}_{(S: T: X: Y)}$ the subvariety $V$ of dimension $4$ defined by
\[V = \overline{\left\{ ((s, t, x, y), (S: T: X: Y)) \st \begin{array}{l} (s, t, x, y) \neq (0, 0, 0, 0),\\ (s: t: x: y) = (S: T: X: Y) \end{array} \right\}},\]
and let $\beta$ be the restriction to $V$ of the projection on $\A^4$. Then
\begin{enumerate}
\item The fibre over a point $(s, t, x, y) \neq (0, 0, 0, 0)$ is the point $((s, t, x, y), (s: t: x: y))$.
\item The fibre over $(0, 0, 0, 0)$ is $\PP^{(1, 1, 2, 3)}$.
\end{enumerate}
Then $V$ is a sort of \emph{weighted blow up} of $\A^4$ at the origin (cf.\ \cite[$\S$4]{C3-f}, where it is called the \emph{$\alpha$-blow up} of $\A^4$), and we claim that the strict transform $\widetilde{W}$ of $W$ in $V$ is a crepant resolution of the singularities of $W$, which introduces a divisor over the singular point.

There is a nice description of $V$ using toric geometry. The affine space $\A^4$ is the toric variety associated to the cone $\Cone(e_1, e_2, e_3, e_4) \subseteq \R^4$, where the $e_i$'s are the standard base vectors. We consider the vector $\alpha = (1, 1, 2, 3)$ and define the fan $\Sigma$ whose maximal cones are
\[\Cone(e_1, e_2, e_3, \alpha), \quad \Cone(e_1, e_2, e_4, \alpha), \quad \Cone(e_1, e_3, e_4, \alpha), \quad \Cone(e_2, e_3, e_4, \alpha):\]
then $V$ is the toric variety associated to the fan $\Sigma$, and $\Sigma \hookrightarrow \Cone(e_1, e_2, e_3, e_4)$ gives the projection $V \longrightarrow \A^4$. From this toric picture, we can also describe $V$ as a quotient space: we have that
\[V = \left( \A^5_{(S, T, X, Y, w)} \smallsetminus \{ S = T = X = Y = 0 \} \right) / \sim\]
where $\sim$ is the equivalence relation induced by the action of $\C^*$ on $\A^5 \smallsetminus \{ S = T = X = Y = 0 \}$
\[\lambda \cdot (S, T, X, Y, w) = (\lambda S, \lambda T, \lambda^2 X, \lambda^3 Y, \lambda^{-1} w).\]
Using the global homogeneous toric coordinates $(S: T: X: Y: w)$ on $V$, the projection on $\A^4$ is
\[(S: T: X: Y: w) \longmapsto (s, t, x, y) = (S w, T w, X w^2, Y w^3)\]
and so we see that over a point $(s, t, x, y) \neq 0$ there is only the point $(s: t: x: y: 1)$, while over $(0, 0, 0, 0)$ we have the divisor $w = 0$ in $V$, which is isomorphic to the weighted projective space $\PP^{(1, 1, 2, 3)}$.\\
We want now to give a description of $\beta: \widetilde{W} \longrightarrow W$ in local coordinates on $V$: from the quotient description we have that $V$ is covered by four local charts, corresponding to the open subsets where $S$, respectively $T$, $X$, $Y$, are non-zero.
\begin{enumerate}
\item[Chart $S \neq 0$] This affine chart is smooth, and the projection to $\A^4$ is
\begin{equation}\label{equation: blow up S neq 0}
(T, X, Y, w) \longmapsto (s, t, x, y) = (w, T w, X w^2, Y w^3):
\end{equation}
the strict transform $\widetilde{W}$ of $W$ via~\eqref{equation: blow up S neq 0} is then
\[\widetilde{W}: Y^2 = X^3 + X + T^6,\]
which is smooth. Observe that we have
\[\left\{ \begin{array}{l}
ds = dw\\
dt = w dT + T dw\\
dx = w^2 dX + 2wX dw\\
dy = w^3 dY + 3w^2Y dw
\end{array} \right.\]
and so the pull back of the residue
\[-\frac{ds \wedge dt \wedge dx}{2 y},\]
which generates $\omega_W$, is the $3$-form
\[-\frac{dT \wedge dX \wedge dw}{2 Y}.\]
This last is a generator for $\omega_{\widetilde{W}}$, and so~\eqref{equation: blow up S neq 0} is crepant. To conclude that $\beta: \widetilde{W} \longrightarrow W$ is crepant we only need to show that $\widetilde{W}$ is smooth, since the part of $\widetilde{W}$ which is not described in this chart is of codimension $2$.
\item[Chart $T \neq 0$] This affine chart is completely analogous to the previous one.
\item[Chart $X \neq 0$] This affine chart is singular, in fact it is isomorphic to the quotient of $\A^4$ by the action
\[(S, T, Y, w) \longmapsto (\pm S, \pm T, \pm Y, \pm w).\]
The projection to $\A^4$ is
\[(S, T, Y, w) \longmapsto (S w, T w, w^2, Y w^3)\]
and so $\widetilde{W}$ is defined in this chart by
\[\widetilde{W}: Y^2 = 1 + S^4 + T^6,\]
which does not pass through the singular point of the chart, and is in fact smooth.
\item[Chart $Y \neq 0$] Also this affine chart is singular, and it is isomorphic to the quotient of $\A^4$ by the action
\[(S, T, X, w) \longmapsto (\zeta_3^2 S, \zeta_3^2 T, \zeta_3 X, \zeta_3 w), \qquad \zeta_3^3 = 1, \zeta_3 \neq 1.\]
The projection to $\A^4$ is
\[(S, T, X, w) \longmapsto (S w, T w, w^2 X, w^3)\]
and so $\widetilde{W}$ is defined in this chart by
\[\widetilde{W}: 1 = X^3 + S^4 X + T^6,\]
which does not pass through the singular point of the chart, and is in fact smooth.
\end{enumerate}
\end{Ex}

\bibliographystyle{alpha}
\bibliography{BiblioSing}

\end{document}